\documentclass[12pt]{amsart}
\usepackage{amsmath}
\usepackage{amsmath}
\usepackage{amssymb}

\title{Operator algebras associated to integral domains}

\author{Benton L. Duncan}

\address{Department of Mathematics\\
North Dakota State University\\
Fargo, North Dakota\\
USA}

\email{benton.duncan@ndsu.edu}



\begin{document}

\theoremstyle{plain}
\newtheorem{thm}{Theorem}
\newtheorem{lem}{Lemma}
\newtheorem{prop}{Proposition}
\newtheorem{cor}{Corollary}

\theoremstyle{definition}
\newtheorem{dfn}{Definition}
\newtheorem*{construction}{Construction}
\newtheorem*{example}{Example}

\theoremstyle{remark}
\newtheorem*{conjecture}{Conjecture}
\newtheorem*{acknowledgement}{Acknowledgements}
\newtheorem{rmk}{Remark}

\begin{abstract} We study operator algebras associated to
integral domains.  In particular, with respect to a set of natural
identities we look at the possible nonselfadjoint operator algebras
which encode the ring structure of an integral domain.  We show that
these algebras give a new class of examples of semicrossed products
by discrete semigroups.  We investigate the structure of these
algebras together with a particular class of representations.
\end{abstract}

\maketitle

Recently, in \cite{Cuntz-Li:2008} and \cite{Li:2009} the notion of a
regular $C^*$-algebra associated to an integral domain was
introduced as a generalization of a construction due to Cuntz,
\cite{Cuntz}.  In these papers the authors associate a $C^*$-algebra
by representing $\ell^2(R)$ and viewing the operators given by the
regular representation $R$ acting on $\ell^2(R)$.  In addition they
show that the $C^*$-algebras so constructed are universal with
respect to a collection of identities that encode information about
the integral domain.

Now one can view an integral domain as an additive group together
with an action on the additive group given by multiplication by
nonzero elements of the ring. This suggests that an important
viewpoint for studying operator algebras associated to integral
domains is through the use of crossed products.  More importantly,
since crossed products are well understood many of the significant
results can be made brief through the technology of crossed
products.

For this paper we wish to investigate the operator algebras with
slightly less restrictive identities imposed by only natural
ring-theoretic constraints.  This gives rise to operator algebras
with a more natural crossed product structure.  However since the
multiplication in a ring need not act as automorphisms on the
additive group, crossed products are not entirely appropriate.  To
avoid this we use the nonselfadjoint operator algebras where
possible. This goes back to a construction of Arveson and Josephson
\cite{Arveson-Josephson} which was generalized by Peters in
\cite{Peters:1984}.  This semicrossed product is a nonselfadjoint
operator algebra which encodes the same dynamics as the crossed
product but does not require that the action on a topological space
be via homeomorphisms.

While one may worry that we lose too much information when we lose
the $*$-structure of the $C^*$-algebra, in recent work
\cite{Davidson-Katsoulis:2006} it was shown that the semicrossed
products of Peters in fact encode the action of a continuous self
map on a topological space in a manner which is unique up to
conjugacy of the map.  This is even true when the map is not a
homeomorphism and hence unlike with $C^*$-algebras the
nonselfadjoint operator algebras can be used as a topological
invariant.

It is these motivating examples which have led us to study the
semicrossed product algebras in the context of integral domains.  In
this paper we have defined the universal operator algebra associated
to an integral domain (note the different conditions we require from
those of Cuntz and Li).  We then study the situation in the case
that our integral domain is a field.  Here the semicrossed product
and the crossed product coincide and we can use standard results for
crossed products to prove facts about the algebra.  After viewing
the case of the integral domain being a field we focus on the
situation where this may not be true.  Here the semicrossed product
technology is necessary, however similar results carry through.
After defining the requisite notion of semicrossed product and
proving some first results in the context of integral domains we
prove some results which show that the algebras thus defined are
distinct from the algebras studied by Peters.  In the last section
we analyze what we call unitary representations of an integral
domain $R$.  We show that every such representation factors through
a regular unitary representation.

We describe some standard notation we intend to use.  If $R$ is an
integral domain we write $Q(R)$ for the field of quotients. We write
$(R,+)$ for the additive group on $R$.  This group is a locally
compact discrete group.  We denote the Pontryagin dual of this group
by $(\widehat{R,+})$ and denote by $\widehat{a}$ the element in the
Pontryagin dual corresponding to an element $a$ in $(R,+)$.

The author would like to thank Jim Coykendall and Sean
Sather-Wagstaff for helpful discussions about and examples of
integral domains.

\section{Universal algebras of integral domains}

Let $R$ be an integral domain.  Given a Hilbert space $\mathcal{H}$
we define an {\em isometric representation of $R$} to be $ \{ S_r
\in B(\mathcal{H}): r \in R^{\times} \}$ a collection of isometries
together with unitaries $ \{ U^n \in B(\mathcal{H}): n \in R \}$
that satisfy the relations:
\begin{enumerate} \item $ S_rS_t = S_{rt} $ for all $ r,
t \in R^{\times} $,

\item $U^n U^m = U^{m+n} $ for all $ m, n \in R$, and

\item $U^nS_r = S_r U^{rn}$ for all $ r \in R^{\times}, n \in R$.
\end{enumerate}

If the $S_r$ are unitaries for all $ r \in R^{\times}$ then we say
that the representation is a {\em unitary representation}.

We present first two examples:

\begin{example} (The regular representation of an integral domain)

Let $\mathcal{H}$ be equal to $\ell^2(R)$, with $e_q$ denoting the
characteristic function of $\{ q \} \subseteq R$. Define operators
$U^n$ and $S_r$ as follows:
\begin{align*} U^n \left( \sum_{q \in R} \zeta_q e_q\right) &=
\sum_{q \in R} \zeta_q e_{q+n}
\\ S_r \left( \sum_{q \in R} \zeta_q e_q\right) &= \sum_{q \in R}
\zeta_q e_{rq}.
\end{align*} It is not difficult to see that $ \{U^n\}$ and $\{
S_r \}$ give rise to an isometric representation of $R$.
\end{example}

\begin{example} (The regular unitary representation of an integral domain)

Let $\mathcal{K}$ be equal to $ \ell^2(Q(R))$, with $e_q$ denoting
the characteristic function of $ \{ q \} \subseteq Q(R)$. We use the
same formulas to define $\widetilde{U^n}$ and $\widetilde{S_r}$.
Notice this time however that for all $r$, $ \widetilde{S_r}$ is
onto and hence a unitary. (To see this notice that
$\widetilde{S_r}(e_{\frac{q}{r}}) = e_q$ for every $q \in Q(R)$).

An important point to notice is that $\mathcal{H} \subseteq
\mathcal{K}$ and further $\mathcal{H}$ is an invariant subspace for
the collections $ \{ \widetilde{S_r} \}$ and $ \{ \widetilde{U^n}
\}$. Further $ S_r = P_{\mathcal{H}} \widetilde{S_r}
|_{\mathcal{H}}$ for all $r$ and $ U^n = P_{\mathcal{H}}
\widetilde{U^n} |_{\mathcal{H}}$.  For this reason we call this
representation the {\em regular unitary representation of $R$}.
\end{example}

We let $A(R)$ be the norm closed operator algebra generated by
unitaries $\{ u^n: n \in R \}$ and isometries $ \{ s_r: r \in
R^{\times} \}$ which is universal for isometric representations of
$R$.  We will denote the elements of $A(R)$ with lower case letters
to distinguish from a representation of $R$ for which we will use
upper case letters.

We notice some initial facts about the algebra $A(R)$.

\begin{lem} $A(R)$ is unital with $u^0 = s_1$. \end{lem}

\begin{proof} Let $\{ U^r, S_r \}$ be an isometric representation of
$R$.  Then notice that $U_0$ is an idempotent unitary, hence $1 =
U_0^*U_0 $. Then $U_0 = 1\cdot U_0 = U_0^*U_0^2 = U_0^*U_0 = 1$. A
similar argument yields the same result for $S_1$.  Since this is
true for an arbitrary isometric representation of $R$, the result
follows for $A(R)$.\end{proof}

\begin{lem} If $r$ is invertible then $s_r$ is a unitary with $s_r^*
= s_{r^{-1}}$  \end{lem}

\begin{proof} This follows from the previous lemma since $S_r
S_{r^{-1}} = S_1 = S_{r^{-1}}S_r$, for any isometric representation
of $R$. \end{proof}

It follows that if $R$ is a field then $A(R)$ is a $C^*$-algebra. In
addition, in this case, the regular representation is a unitary
representation. In fact we have the following:

\begin{prop} $R$ is a field if and only if every isometric representation is
a unitary representation. \end{prop}

\begin{proof} This comes from the fact that for the regular
representation $(S_r)^*$ is in the algebra if and only if $r$ is
invertible.\end{proof}

Finally we can see that $A(R)$ is functorial for ring monomorphisms
since any isometric representation of a ring $R_2$ will give rise to
an isometric representation of a ring $R_1$ if there is a ring
monomorphism from $R_1 $ into $R_2$..

\begin{prop} $A(R)$ is functorial in the sense that if $\pi: R_1
\rightarrow R_2$ is a unital ring monomorphism then there is an
induced completely contractive representation $\tilde{\pi}: A(R_1)
\rightarrow A(R_2)$ \end{prop}

\section{The universal $C^*$-algebra for a field}

We now analyze the case where $R$ is a field.  Here any isometric
representation is, in fact, a unitary representation.  We let
$(R,+)$ denote the additive group in $R$. Notice that $R^{\times}$
acts on $C^*(R,+)$ as $*$-automorphisms via the mapping
$\alpha_{\lambda}(U^n) = U^{\lambda n}$ where $U^n$ is the unitary
in $C^*(R,+)$ corresponding to $n \in (R,+)$.  This allows us to
rewrite $C^*(R)$ as a crossed product.

\begin{prop} Let $R$ be a field, then $A(R) \cong C^*(R,+)
\rtimes R^{\times}$. \end{prop}

\begin{proof} We begin by noting (see \cite[II.10.3.10]{Blackadar:2006})
that since $C^*(R,+)$ is unital and $R^{\times}$ is discrete we have
$C^*(R,+) \subset C^*(R,+) \rtimes R^{\times}$ via a representation
$\pi_0$ and there is a natural map $\rho_0: R^{\times} \rightarrow
C^*(R,\times) \rtimes R^{\times}$. Together $(\pi_0,\rho_0)$ give
rise to a covariant representation of the triple $(C^*(R,+),
R^{\times}, \alpha)$.

Now analyzing the covariance conditions that define $C^*(R,+)
\rtimes R^{\times}$ we see that $\rho_0(r) \pi_0(n) \rho_0(r)^* =
\pi_0(rn)$.  Now for each $n$, $\pi_0(n)$ is a unitary and for each
$r$, $\rho_0(r)$ is a unitary and hence the natural covariant
representation $\pi_0 \rtimes \rho_0$ gives rise to a unitary
representation of $R$, and hence there is a completely contractive
representation $\iota :A(R) \rightarrow C^*(R,+) \rtimes
R^{\times}$.

Next notice that any unitary representation of $A(R)$ gives rise to
a covariant representation of $(C^*(R,+),R^{\times}, \alpha)$ and
hence $ \| x \| \leq \| \iota(x) \|$ so that $\iota$ is faithful.
\end{proof}

Since $C^*(R,+)$ and $R^{\times}$ are both abelian we have that the
universal norm on the crossed product $C^*(R,+) \rtimes R^{\times}$
coincides with the reduced norm \cite[Theorem 7.13]{Williams:2007}.
In addition, we have that the algebra $A(R)$ is nuclear
\cite[Corollary 7.18]{Williams:2007}, when $R$ is a field.  We next
analyze the regular representation of $A(R)$ where $R$ is a field.

\begin{prop} Let $R$ be a field, then the regular
representation of $A(R)$ is faithful. \end{prop}

\begin{proof}  This follows from analyzing the regular
representation of the algebra $C^*(R,+) \rtimes_{\alpha}
R^{\times}$, which is faithful since $R^{\times}$ is amenable. In
effect, we take the left regular representation of $C^*(R,+)$ acting
on $L^2(R,+)$ and add to this the action of $R^{\times}$ via
$*$-automorphisms. This is exactly the construction of the regular
representation of $R$ and hence the two coincide.\end{proof}

Other facts about $A(R)$ can also be explained using the crossed
product machinery.

\begin{prop} For a field $R$ the algebra $A(R)$ is not simple. \end{prop}

\begin{proof}  Notice that $C^*(R,+) = C(\widehat{R,+})$ where
$(\widehat{R,+})$ is the Pontryagin dual of the locally compact
abelian group $(R,+)$. Notice that the $*$-automorphisms
$\alpha_{\lambda}$ induce a homeomorphism
$\widehat{\alpha_{\lambda}}$ on $(\widehat{R,+})$ which has a fixed
point for each $\lambda$; in particular,
$\widehat{\alpha_{\lambda}}(\widehat{0}) = \widehat{0})$ for all
$\lambda$.  It follows that there is a nontrivial invariant ideal in
$C^*(R,+)$ for the action by $R^{\times}$ and hence, see
\cite[Section 3.5]{Williams:2007} there is a nontrivial induced
ideal in $C^*(R,+) \rtimes R^{\times}$.
\end{proof}

We can, however completely describe the ideal structure of $A(R)$ in
the case of a field.

\begin{thm} $A(R)$ is $*$-isomorphic to $\mathbb{C} \oplus A$
where $A$ is a simple $C^*$-algebra.  In fact, $A$ is $*$-isomorphic
to $C_0((\widehat{R,+}) \setminus \{ \widehat{0} \}) \rtimes
R^{\times}$.
\end{thm}

\begin{proof} We use the nontrivial invariant ideal from the previous
proposition.  In particular, since $C^*(R,+) = C(\widehat{R,+})$ let
$\pi$ be the multiplicative linear function given on
$C(\widehat{R,+})$ by evaluation at $\widehat{0}$.  Further, if we
let $\widehat{\alpha_{\lambda}}$ be the induced homeomorphism on
$(\widehat{R,+})$ given by the $*$-automorphism $\alpha_{\lambda}$
for all $ \lambda \in R^{\times}$.  The induced representation,
$\pi$ is a multiplicative linear functional and hence has range
$\mathbb{C}$. Hence, $A(R) \cong \mathbb{C} \oplus \ker \pi$.  We
now wish to describe $\pi$. So let $\sigma: C_0(\widehat{R,+})
\rightarrow C_0( (\widehat{R,+}) \setminus \{ 0 \} )$ be the
restriction mapping. Further, if $\lambda \in R^{\times}$ then $
\widehat{\alpha_{\lambda}}$, the homeomorphism on $\widehat{R,+}$
induced by the automorphism $\alpha_{\lambda}$, then the range of
$\sigma$ is invariant under $\widehat{\alpha_{\lambda}}$ and hence
there is a map $\tau: C_0((\widehat{R,+}) \setminus \{ 0 \}) \rtimes
R^{\times} \rightarrow \ker \pi$.  But since $R^{\times}$ acts
transitively on $(\widehat{R,+})$ the crossed product
$C_0((\widehat{R,+}) \setminus \{ 0 \}) \rtimes R^{\times}$ is
simple and hence the map $\tau$ must be an isomorphism.\end{proof}

\section{Semicrossed products for discrete semigroups}

The preceding construction suggests that for non-field integral
domains the crossed product may be replaced by a semicrossed
product.  We quickly outline the relevant construction referring to
\cite{Peters:1984} for motivation and to \cite{Duncan-Peters} for
more information about this semicrossed product.

Given a compact Hausdorff space $X$ we say that a semigroup $S$ acts
on $X$ via continuous maps if for each $s \in S$ there is a
continuous map $\tau_s: X \rightarrow X$ with $ \tau_{s} \circ
\tau_{t} = \tau_{st}$.  If $S$ is unital with identity $0$ we will
assume that $ \tau_0$ is the identity map.  Say that a pair $(\pi,
S_t)$ is an isometric covariant representation of $(X, S, \tau_s)$
if $\pi$ is a representation of $C(X)$ on a Hilbert space
$\mathcal{H}$ and for each $t \in S$, $S_t$ is an isometry in
$B(\mathcal{H})$ such that $ S_t \pi(f(x)) = \pi(f(\tau_t(x)))S_t$
for all $ x \in X$.

It is not hard to see that given the triple $(X,S, \tau_s)$, there
is a nontrivial isometric covariant representation.  The
construction follows in the same manner as in \cite{Peters:1984}, we
only outline the idea here.  Let $\mathcal{H} = \ell^2(X,S)$ where
this latter Hilbert space is sequences indexed over elements of $S$
with entries from $x$, with canonical basis $\{ e_s \}$.  Define
$\pi: C(X) \rightarrow B(\mathcal{H})$ by $ \pi(f(x)) =
(f(\tau_s(x)))_{s \in S}$.  Then set $S_t (e_s) = e_{st}$ and extend
by linearity.  Then $(\pi,S_t)$ is an isometric covariant
representation of $(X,S,\tau_s)$.

We say that the universal operator algebra generated by all
isometric covariant representations of $(X,S,\tau_s)$ is the
semicrossed product of $X$ by $S$ via $\tau$.  We denote this
algebra as $C(X) \rtimes_{\tau} S$.

As examples notice that if $\alpha$ is a single endomorphism of a
$C^*$-algebra then we are in the situation described in Peter's
original work \cite{Peters:1984}, where the semigroup is
$\mathbb{Z}^+$.  For an example on the opposite end of the spectrum
we can view the semicrossed product of
\cite{Davidson-Katsoulis:2007} as a semicrossed product where the
monoid is the free monoid on $n$ generators.

Returning to an integral domain $R$, we let $\alpha_r$ be the
$*$-endomorphism of $C^*(R,+)$ induced by the group endomorphism
given by left multiplication by $r$.  This gives a map from
$R^{\times}$ into the set of $*$-endomorphisms of $C^*(R,+)$. Notice
that since we are in an integral domain each of these endomorphisms
is injective.  However they are only surjective when $r$ is a unit.

\begin{prop} The algebra $A(R)$ is completely isometrically
isomorphic to $C^*(R,+) \rtimes_{\alpha} R^{\times}$. \end{prop}

\begin{proof} We will show that any isometric representation of
$R$ gives rise to an isometric covariant representation of the pair
$(C^*(R,+) \rtimes_{\alpha} R^{\times})$ and vice-versa, hence the
two algebras will be completely isometrically isomorphic.

So let $ \{ U^n: n \in R \}$ and $ \{ S_r \in R^{\times} \}$ be an
isometric representation of $R$.  Then the map $n \mapsto U^n$ gives
rise to a representation of $C^*(R,+)$, call it $\pi$.  Further the
isometries $S_r$ will satisfy $ S_rU^nS_r^* = U^{rn}$ and hence the
pair $(\pi, \{S_r \})$ will be an isometric covariant representation
of $ (C^*(R, +), R^{\times}, \alpha)$.

Let $(\pi,\{S_r\})$ be an isometric covariant representation of the
algebra $ (C^*(R,+), R^{\times}, \alpha)$.  Then $\{ \pi(u^n) \}$ is
a collection of unitaries and $ \{ S_r \}$ is a collection of
isometries that trivially satisfy the first two conditions of an
isometric representation of $R$.  Further $S_rU^nS_r^* = U^{rn}$ so
that $S_rU^n = U^{rn}S_r$ for all $r \in R^{\times}$ so that we have
an isometric representation of $R$.\end{proof}

Notice that in the case of the algebra $A(R)$ for an integral domain
$R$ the semigroup $R^{\times}$ will always be commutative with no
torsion. In addition the semigroup $R^{\times}$ can be viewed as a
spanning cone for the group $Q(R)^{\times}$ (see \cite[Page
60]{Paulsen:2002} for the definition of a spanning cone for a
group).

We can actually improve our characterization of $A(R)$ as a
semicrossed product by looking at a more tractable semigroup. Let
$U(R)$ denote the group of units in $R$. Now $R^{\times}$ is a
commutative monoid which contains $U(R)$ as a normal submonoid. We
let $M(R)$ denote the monoid $R^{\times}/U(R)$. Notice that
$R^{\times} \subseteq Q(R)^{\times}$ and further that $M(R)
\subseteq Q(R)^{\times}/U(R)$, this latter group we call $G(R)$.

For $u \in U(R)$ let $\alpha_u: C^*(R,+) \rightarrow C^*(R,+)$ be
the $*$-automorphism induced by the automorphism of $(R,+)$ that
corresponds to left multiplication by $u$. Next for $r \in M(R)$
define a covariant representation of the triple $(C^*(R,+), U(R),
\alpha)$ by $\beta_r(U^n) = U^{nr}$, $\rho(S_u) = S_u$.  This
covariant representation induces a $*$-endomorphism of $C^*(R,+)
\rtimes_{\alpha} U(R)$.  Hence, $\beta$ gives rise to a map from
$M(R)$ into the set of $*$-endomorphisms of $C^*(R,+)
\rtimes_{\alpha} U(R)$.

\begin{thm} The algebra $A(R)$ is completely isometrically
isomorphic to $ (C^*(R,+) \rtimes_{\alpha} U(R)) \rtimes_{\beta}
M(R) $, and the diagonal algebra $A(R) \cap A(R)^* = (C^*(R,+)
\rtimes_{\alpha} U(R))$. \end{thm}

\begin{proof} We will show that any isometric representation of
$R$ gives rise to an isometric covariant representation of the pair
\[(C^*(R,+) \rtimes_{\alpha} U(R), M(R))\] and vice-versa, hence the
two algebras will be completely isometrically isomorphic.

So let $\{ U^n: n \in R \}$ and $ \{ S_r: r \in R^{\times} \}$ be an
isometric representation of $R$.  Then define a covariant
representation of $(C^*(R,+), U(R), \alpha)$ by $ n \mapsto U^n$ and
$r \mapsto S_r$ for all $r \in U(R)$.  This yields a representation
$\pi$ of $C^*(R,+) \rtimes_{\alpha}U(R)$.  Next notice that $
S_{\lambda} \pi(U^n) = \pi(U^{\lambda n}) S_{\lambda}$ for all
$\lambda \in M(R)$, and $ S_{\lambda} S_r = S_r S_{\lambda}$ for all
$ r \in U(R), \lambda \in M(R)$.  Hence we have an isometric
covariant representation of $(C^*(R,+) \rtimes_{\alpha} U(R), M(R),
\beta)$.

Finally we take an isometric covariant representation $(\pi, S)$ of
the triple $(C^*(R,+) \rtimes_{\alpha} U(R), M(R), \beta)$.  Define
$U^n = \pi(n)$ and $S_r = \pi(S_r)$ if $r \in U(R)$, else $S_r =
S_r$. This gives rise to an isometric representation of $R$.

The last result follows from Corollary 2 of \cite{Duncan-Peters}.
\end{proof}

Notice that if $R$ is not a field $U(R)$ does not act transitively
on the nonunital subalgebra of $C^*(R,+)$, as in the case of a
field.  In fact we have the following fact.

\begin{cor} The diagonal is isomorphic to $\mathbb{C} \oplus A$
where \[A \cong C_0((\widehat{G,+}) \setminus \widehat{0}) \times
U(R). \] Further, $A$ is simple if and only if $R$ is a field.
\end{cor}

\begin{proof} That $A$ is not simple follows from the fact that
$U(R)$ does not act transitively on $C_0(\widehat{C^*(G,+)}
\setminus 0)$ unless $R^{\times} = U(R)$.\end{proof}

We now prove some other facts about the relationship between the
integral domain $R$ and the structure of the algebra $A(R)$.

\begin{prop} $A(R) \cong C(X) \rtimes S$, where $S$ is a monoid with
no nontrivial invertible elements if and only if the identity of $R$
is the only unit.  Further if $U(R) = \{ 1 \}$ then $M(R)$ will not
be finitely generated.
\end{prop}

\begin{proof} If the identity of $R$ is the only unit, then we
have \[C^*(R,+) \times U(R) \cong C^*(R,+)\] and $M(R)$ has no
nontrivial invertible elements, else $M(R) \cap U(R) \neq \{ 1 \}$.

Notice that if $A(R) \cong C(X) \times S$ where $S$ is a monoid with
no nontrivial invertible elements then the diagonal algebra $A(R)
\cap A(R)^* = C(X)$.  However, if $x \in U(R)$ with $x \neq 1$, then
$x \not\in M(R)$ and hence $S_x \in C^*(R,+) \times U(R) = A(R) \cap
A(R)^*$. But notice that $S_xU^n \neq U^nS_x$ unless $ x = 1$ and
hence $C^*(R,+) \times U(R)$ is not commutative.

Assume now that $U(R) = \{ 1 \}$ and $M(R)$ is finitely generated by
the set $ \{ x_1, x_2, \cdots, x_n \}$, then $1 + x_1x_2 \cdots x_n
\not\in U(S)$ so \[ 1+ x_1x_2 \cdots x_n = x_{1}^{k_1}x_2^{k_2}
\cdots x_n^{k_n}\] with at least one $k_j \neq 0$.  We will assume
without loss of generality that $k_1 \neq 0$.  Then $1 = x_1(x_2 x_3
\cdots x_n - x_1^{k_1-1}x_2^{k_2} \cdots x_n^{k_n})$ which implies
that $x_1$ is a unit yielding a contradiction. \end{proof}

As a corollary we have the following.

\begin{prop} If $R$ is a unique factorization domain and
$A(R) \cong A \times \mathbb{Z}^+$ where $A$ is a $C^*$-algebra,
then $A$ is not commutative and $U(R)$ is not trivial. \end{prop}

\begin{proof}  We will assume that $A$ is commutative and hence
$U(R)$ is trivial.  In particular $A = C^*(R,+) =
C_0(\widehat{R,+})$. Now let $x_1$ and $x_2$ be two irreducible
elements of $M(R)$, then define two two-dimensional nest
representations of $A(R)$ by \begin{align*} \pi_i(f) & =
\begin{bmatrix} f(\widehat{x_i}) & 0 \\ 0 & f(\widehat{1})
\end{bmatrix}, \mbox{ for } f \in C_0(\widehat{R,+}) \\
\pi_i(S_{x_i}) & = \begin{bmatrix} 0 & 1 \\ 0 & 0 \end{bmatrix} \\
\pi_i(S_r) & = 0 \mbox{ for } r \neq x_i. \end{align*}

It follows from the description of the two-dimensional nest
representations of $A \times \mathbb{Z}^+$, see
\cite{Davidson-Katsoulis:2006}, that $x_1 = x_2$ which contradicts
the fact that $M(R)$ must be infinitely generated.
\end{proof}

It follows that if $R$ is a unique factorization domain $A(R)$ is
never a semicrossed product in the sense of \cite{Peters:1984} and
hence this collection of algebras presents a unique type of
semicrossed product.

\section{Unitary representations of $R$}

Finally we wish to analyze the unitary representations of $A(R)$.

\begin{lem}\label{inclusion} There is a canonical completely contractive
representation $i: A(R) \rightarrow C^*(Q(R))$. \end{lem}

\begin{proof} As $R \subseteq Q(R)$ the inclusion map
provides an isometric representation of $R$ inside $C^*(Q(R))$ and
hence the induced map on $A(R)$ is completely contractive.
\end{proof}

\begin{prop}\label{extension} Let $\pi: A(R) \rightarrow B(H)$ be a unitary
representation.  There exists a $*$-representation $ \tau_{\pi}:
C^*(Q(R)) \rightarrow C^*(\pi(A(R))$ which is onto and satisfies $
\tau_{\pi} \circ i (x) = \pi(x)$ for all $x \in A(R)$. \end{prop}

\begin{proof} Since $\pi$ is a unitary representation we know that $
\pi (s_r) = T_r$ is a unitary for all $r \in R^{\times}$.  For all
$\left[ \frac{p}{q} \right] \in Q(R)^{\times}$ we define
$\tilde{\pi} \left(s_{[\frac{p}{q}]}\right) = T_p T_q^*$.  We also
define $ \tilde{\pi} \left( u^{[\frac{p}{q}]} \right) = T_q V^p
T_q^*$.  We need only show that the unitaries $T_pT_q^*$ and $T_q
V^p T_q^*$ satisfy the relations for $Q(R)$ and hence the induced
representation $\tau_{\pi}$ will be the required $*$-representation.

Notice first that $ T_p T_q = T_q T_p$ and $T_p^*T_q^* = T_q^*T_p^*$
since $\pi$ is a unitary representation of $A(R)$.  It then follows
that \begin{align*}T_p T_q^* &= T_q^*T_qT_pT_q^* \\ &=
T_q^*T_pT_qT_q^* \\ &= T_q^*T_p \end{align*} for all $ q, p \in
R^{\times}$. It follows that $\tilde{\pi} \left(s_{[ \frac{p_1}{q_1}
]}\right) \tilde{\pi} \left(s_{[ \frac{p_2}{q_2} ]}\right) =
\tilde{\pi} \left(s_{[ \frac{p_1p_2}{q_1q_2} ]}\right)$ for all $ [
\frac{p_1}{q_1} ]$ and $ [ \frac{p_2}{q_2}]$ in $Q(R)^{\times}$.

Next notice that $V^p T_q = T_q V^{pq}$ and $ T_q^* V^p =
V^{pq}T_q^*$ and hence \begin{align*} T_{q_1}V^{p_1}T_{q_1}^*
T_{q_2}V^{p_2}T_{q_2}^* &=
T_{q_1}T_{q_2}V^{p_1q_2}V^{q_1p_2}T_{q_1}^*T_{q_2}^* \\ & =
T_{q_1q_2} V^{p_1q_2 + q_1p_2}T_{q_1q_2}.\end{align*}  In other
words $\tilde{\pi} \left(u^{[\frac{p_1}{q_1}]} \right) \tilde{\pi}
\left( u^{[\frac{p_2}{q_2}]} \right) = \tilde{\pi} \left( u^{[
\frac{p_1}{q_1}] + [\frac{p_2}{q_2}]} \right) $.

Next we have that \begin{align*}\tilde{\pi} \left(
u^{[\frac{p_1}{q_1}]}\right) \tilde{\pi} \left(s_{[\frac{p_2}{q_2}]}
\right) &= T_{q_1}V^{p_1}T_{q_1}^* T_{p_2}T_{q_2}^* \\& =
T_{p_2}T_{q_2}^* T_{q_1} T_{q_2} V^{p_1p_2} T_{q_2}^* T_{q_1}^* \\&
= T_{p_2}T_{q_2}^* T_{q_1q_2} V^{p_1p_2} T_{q_1q_2}^* \\ &=
\tilde{\pi} \left( s_{[ \frac{p_2}{q_2} ]} \right) \tilde{\pi}
\left( u^{ [ \frac{p_1}{q_1} ][\frac{p_2}{q_2}]} \right).
\end{align*} Hence the $C^*$-algebra generated by the $T_p$ and $V^n$
satisfies the relations for $Q(R)$.  We call the induced
representation $\tau_{\pi}$.

Finally we note that $ \tau_{\pi} \circ i(s_p) = \tau_{\pi}
\left(s_{[\frac{p}{1}]}\right) = T_pT_1^* = T_p$ for all $ p \in
Q^{\times}$ and $ \tau_{\pi} \circ i(u^n) = \tau_{\pi}
\left(u^{[\frac{n}{1}]}\right) = T_1V_nT_1^* = V_n$ for all $ n \in
R$ and hence $ \tau_{\pi} \circ i (x) = \pi(x)$ for all $x \in
A(R)$.\end{proof}

It would follow that if every isometric representation of $R$
dilated to a unitary representation (as for example the regular
representation does), then we could identify the $C^*$-envelope of
$A(R)$ as a crossed product, since the canonical representation $i$
would be completely isometric.  We do not think this is likely since
this does not even work in the case of $C(X) \rtimes_{\alpha}
\mathbb{Z}^+$ where $\alpha$ is a non-surjective continuous mapping,
see \cite{Davidson-Katsoulis:2007}.




\bibliographystyle{plain}

\end{document}